\documentclass[a4paper,leqno]{amsart}

\usepackage{amssymb,amsthm,amsmath}
\usepackage[nobysame]{amsrefs}

\usepackage{ifthen}

\newcounter{mylisti} \newcounter{mylistii}
\newcounter{nest}
\newcommand{\defaultlabel}{}

\newenvironment{mylist}[1]{%
  \addtocounter{nest}{1}
  \ifthenelse{\value{nest}=1}{%
    \renewcommand{\defaultlabel}{(\roman{mylisti})\hfill}}{%
    \renewcommand{\defaultlabel}{(\alph{mylistii})\hfill}}
  \begin{list}{\defaultlabel}{%
      \ifthenelse{\value{nest}=1}{\usecounter{mylisti}}{%
        \usecounter{mylistii}}
      
      \addtolength{\itemsep}{0.5ex}
      \settowidth{\labelwidth}{#1}
      \setlength{\leftmargin}{\labelwidth}
      \addtolength{\leftmargin}{\labelsep}}}{\addtocounter{nest}{-1}
\end{list}}

\newenvironment{mylista}[1]{%
  \addtocounter{nest}{1}
  \ifthenelse{\value{nest}=1}{%
    \renewcommand{\defaultlabel}{(\alph{mylisti})\hfill}}{%
    \renewcommand{\defaultlabel}{(\roman{mylistii})\hfill}}
  \begin{list}{\defaultlabel}{%
      \ifthenelse{\value{nest}=1}{\usecounter{mylisti}}{%
        \usecounter{mylistii}}
      
      \addtolength{\itemsep}{0.5ex}
      \settowidth{\labelwidth}{#1}
      \setlength{\leftmargin}{\labelwidth}
      \addtolength{\leftmargin}{\labelsep}}}{\addtocounter{nest}{-1}
\end{list}}

\usepackage{accents}

\newcommand{\be}{\ensuremath{\mathbb E}}

\newcommand{\bn}{\ensuremath{\mathbb N}}

\newcommand{\br}{\ensuremath{\mathbb R}}

\newcommand{\cA}{\ensuremath{\mathcal A}}
\newcommand{\cB}{\ensuremath{\mathcal B}}

\newcommand{\cI}{\ensuremath{\mathcal I}}
\newcommand{\cJ}{\ensuremath{\mathcal J}}
\newcommand{\cK}{\ensuremath{\mathcal K}}
\newcommand{\cL}{\ensuremath{\mathcal L}}

\newcommand{\cS}{\ensuremath{\mathcal S}}
\newcommand{\cT}{\ensuremath{\mathcal T}}

\usepackage{mathrsfs}

\newcommand{\cJt}{\ensuremath{\widetilde{\cJ}}}

\newcommand{\lt}{\ensuremath{\tilde{l}}}
\newcommand{\mt}{\ensuremath{\widetilde{m}}}

\newcommand{\xt}{\ensuremath{\tilde{x}}}
\newcommand{\yt}{\ensuremath{\tilde{y}}}

\newcommand{\Et}{\ensuremath{\widetilde{E}}}
\newcommand{\Ft}{\ensuremath{\widetilde{F}}}

\newcommand{\Tt}{\ensuremath{\widetilde{T}}}

\newcommand{\abs}[1]{\lvert #1\rvert}
\newcommand{\bigabs}[1]{\big\lvert #1\big\rvert}
\newcommand{\Bigabs}[1]{\Big\lvert #1\Big\rvert}
\newcommand{\biggabs}[1]{\bigg\lvert #1\bigg\rvert}
\newcommand{\sabs}[1]{\left\lvert #1\right\rvert}

\newcommand{\cl}[1]{\overline{#1}}

\newcommand{\dist}{\ensuremath{\mathrm{dist}}}

\newcommand{\infin}[1]{{[#1]}^{\omega}}

\newcommand{\ip}[2]{\ensuremath{\langle #1,#2\rangle}}

\newcommand{\intp}[1]{\ensuremath{\lfloor #1\rfloor}}

\newcommand{\ceil}[1]{\ensuremath{\lceil #1\rceil}}

\newcommand{\join}{\vee}

\newcommand{\meet}{\wedge}

\newcommand{\norm}[1]{\lVert #1\rVert}
\newcommand{\bignorm}[1]{\big\lVert #1\big\rVert}
\newcommand{\Bignorm}[1]{\Big\lVert #1\Big\rVert}
\newcommand{\biggnorm}[1]{\bigg\lVert #1\bigg\rVert}
\newcommand{\snorm}[1]{\left\lVert #1\right\rVert}

\newcommand{\spn}{\ensuremath{\mathrm{span}}}
\newcommand{\cspn}{\ensuremath{\overline{\mathrm{span}}}}

\newcommand{\symdif}{\bigtriangleup}

\newcommand{\co}{\mathrm{c}_0}

\newcommand{\continuum}{\mathfrak{c}}

\newcommand{\vare}{\varepsilon}
\newcommand{\varf}{\varphi}

\renewcommand{\geq}{\geqslant}
\renewcommand{\leq}{\leqslant}

\newcommand{\ds}{\displaystyle}

\newcommand{\ts}{\textstyle}

\newcommand{\ie}{\textit{i.e.,}\ }

\newcommand{\cf}{\textit{cf.}\ }

\newtheorem{thm}{Theorem}

\newtheorem*{mainproblem*}{Problem}

\newtheorem{lem}[thm]{Lemma}
\newtheorem{prop}[thm]{Proposition}
\newtheorem{cor}[thm]{Corollary}
\newtheorem{problem}[thm]{Problem}

\theoremstyle{definition}

\theoremstyle{remark}

\newtheorem*{rem}{Remark}

\begin{document}

\allowdisplaybreaks

\title[Banach spaces with $2^{\continuum}$ closed operator
  ideals]{Banach spaces for which the space of operators has
  $2^{\continuum}$ closed ideals}

\author{D.~Freeman}
\address{Department of Mathematics and Statistics, St Louis
 University, St Louis, MO 63103, USA}
\email{daniel.freeman@slu.edu}

\author{Th.~Schlumprecht}
\address{Department of Mathematics, Texas A\&M University, College
  Station, TX 77843, USA and Faculty of Electrical Engineering, Czech
  Technical University in Prague,  Zikova 4, 166 27, Prague}
\email{schlump@math.tamu.edu}

\author{A.~Zs\'ak}
\address{Peterhouse, Cambridge, CB2 1RD, UK}
\email{a.zsak@dpmms.cam.ac.uk}

\keywords{Operator ideals, Rosenthal's $X_{p,w}$ space, factorizing
  operators}
\subjclass[2010]{Primary: 47L20. Secondary: 47B10, 47B37}
\thanks{The first author was supported by grant 353293 from the Simons
  Foundation and the third author's research was supported by NSF
  grant DMS-1764343.}
\begin{abstract}
  We  formulate general conditions which imply that $\cL(X,Y)$, the
  space of operators from a Banach space $X$ to a Banach space $Y$,
  has $2^{\continuum}$ closed ideals where $\continuum$ is the
  cardinality of the continuum. These results are applied to
  classical sequence spaces and Tsirelson type spaces. In
  particular, we prove that the cardinality of the set of closed
  ideals in $\cL(\ell_p\oplus\ell_q)$ is exactly $2^{\continuum}$
  for all $1<p<q<\infty$, which in turn gives an alternate proof of
  the recent result of Johnson and Schechtman that $\cL(L_p)$ also
  has $2^\continuum$ closed ideals for $1<p\neq 2<\infty$.
\end{abstract}

\maketitle

\section{Introduction}

Given Banach spaces $X$ and $Y$, we call a subspace $\cJ$ of the space
$\cL(X,Y)$ of bounded operators an \emph{ideal }if $ATB\in \cJ$ for
all $A\in\cL(Y)$, $T\in\cJ$ and $B\in\cL(X)$. In the case that $X=Y$,
this coincides with the standard algebraic definition of $\cJ$ being 
an ideal in the algebra of bounded operators $\cL(X)$. In this paper
we will only be considering closed ideals. For example, if $X$ and $Y$
are any Banach spaces, then the space of compact operators from $X$ to
$Y$ and the space of strictly singular operators from $X$ to $Y$ are
both closed ideals in $\cL(X,Y)$.  In the case of $X=Y=\ell_p$, the
compact and strictly singular operates coincide and they are the only
closed ideal in $\cL(\ell_p)$ other than the trivial cases of $\{0\}$
and the entire space $\cL(\ell_p)$. For $p\neq 2$, the situation for
$L_p$ is very different from $\ell_p$. If $X$ contains a complemented
subspace $Z$ such that $Z$ is isomorphic to $Z\oplus Z$, then the
closure of the set of operators in $\cL(X)$ which factor through $Z$
is a closed ideal, and moreover the map that associates this closed
ideal with the isomorphism class of $Z$ is injective. In the case
$1<p<\infty$ with $p\neq 2$, there  are infinitely many (even
uncountably many) distinct complemented subspaces of $L_p$ which are
isomorphic to their square~\cite{bourgain-rosenthal-schechtman:81},
and thus there are infinitely many distinct closed ideals in
$\cL(L_p)$.

Obviously, constructing infinitely many closed ideals for
$\cL(\ell_p\oplus\ell_q)$ or $\cL(\ell_p\oplus \co)$ with
$1\leq p<q<\infty$ requires different techniques than just considering
complemented subspaces, and it was a long outstanding question from
Pietsch's book~\cite{pietsch:80} whether these spaces have infinitely
many distinct closed ideals. For the cases $1\leq p<q<\infty$, the
closures of the set of operators which factor through $\ell_p$ and the
operators which factor through $\ell_q$ are distinct closed ideals
(indeed, the only maximal ideals) in $\cL(\ell_p\oplus\ell_q)$, and
all other proper closed ideals in $\cL(\ell_p\oplus\ell_q)$ correspond
to closed ideals in $\cL(\ell_p,\ell_q)$.  Progress on constructing
new ideals in $\cL(\ell_p,\ell_q)$ proceeded through building finitely
many ideals at a time (see~\cite{sari-schlump-tomczak-troitsky:07}
and~\cite{schlump:12}) until it was shown using finite-dimensional
versions of Rosenthal's $X_{p,w}$ spaces that there is chain of a
continuum of distinct closed ideals in $\cL(\ell_p,\ell_q)$ for all
$1<p<q<\infty$~\cite{sz:18}. For $1<p<\infty$,
$p\neq 2$, $\ell_p\oplus \ell_2$ is isomorphic to a complemented
subspace of $L_p$, and thus there are at least a continuum of
closed ideals in $\cL(L_p)$.  Other new constructions for building
infinitely many closed ideals soon followed. Wallis
observed~\cite{wallis:15} that the techniques of~\cite{sz:18} extend 
to prove the existence of a chain of a continuum of closed ideals for
$\cL(\ell_p,\co)$ in the range $1<p<2$, and for $\cL(\ell_1,\ell_q)$
in the range $2<q<\infty$. Then, using ordinal indices, Sirotkin and
Wallis proved that there is an $\omega_1$-chain of closed ideals in
$\cL(\ell_1,\ell_q)$ for $1<q\leq\infty$ as well as in
$\cL(\ell_1,\co)$ and in $\cL(\ell_p,\ell_\infty)$ for
$1\leq p<\infty$~\cite{sirotkin-wallis:16}. Using matrices with the
Restricted Isometry Property (RIP), both
chains and anti-chains of a continuum of distinct closed ideals were
constructed in $\cL(\ell_p,\co)$, $\cL(\ell_p,\ell_\infty)$, and
$\cL(\ell_1,\ell_p)$ for all $1<p<\infty$~\cite{fsz:17}.

Recently, using the infinite-dimensional $X_{p,w}$ spaces of Rosenthal
and almost disjoint sequences of integers, Johnson and Schechtman
proved that there are $2^{\continuum}$ distinct closed ideals in
$\cL(L_p)$ for $1<p<\infty$ with $p\neq 2$~\cite{js:20}. In
particular, the cardinality of the set of closed ideals in $\cL(L_p)$
is exactly $2^{\continuum}$.

The goal for this paper is to present a general method for proving
when $\cL(X,Y)$ contains $2^{\continuum}$ distinct closed ideals for
some Banach spaces $X$ and $Y$ with unconditional finite-dimensional
decompositions (UFDD).  Given a collection of operators
$(T_N)_{N\in\infin{\bn}}$ from $X$ to $Y$ indexed by the set of all
infinite subsets of the natural numbers, we give sufficient conditions
for there to exist an infinite subset $L$ of $\bn$ so that if
$\cS\subset\infin{L}$ is a set of pairwise almost disjoint subsets of
$L$, then for all $\cA,\cB\subset S$, if $M\in\cA\setminus\cB$, the
operator $T_M$ is not contained in the smallest closed ideal
containing $\{T_N:\,N\in\cB\}$. Thus, $\cL(X,Y)$ contains
$2^\continuum$ closed ideals.  We are able to apply this method to
prove in particular that the cardinality of the set of closed ideals
in $\cL(\ell_p\oplus\ell_q)$ is exactly $2^{\continuum}$ for all
$1<p<q<\infty$. It follows at once that $\cL(L_p)$ contains exactly
$2^\continuum$ closed ideals for $1<p\neq 2<\infty$, and thus we have
another proof of the aforementioned result of Johnson and
Schechtman~\cite{js:20}. It is worth pointing out that they construct
closed ideals using operators that are not even strictly singular (and
on the other hand, their ideals do not contain projections onto
non-Hilbertian subspaces). By contrast, our $2^\continuum$ closed
ideals are small in the sense that they consist of finitely strictly
sigular operators.

In~\cite{fsz:20} it was shown that there are
$2^{\continuum}$ distinct closed ideals in $\cL(\ell_p,\co)$, 
$\cL(\ell_p,\ell_\infty)$ and $\cL(\ell_1,\ell_p)$ for all
$1<p<\infty$. In this article, we will show that this result can also
be obtained by our general construction.

Although our initial goals were to construct closed ideals between
classical Banach spaces, the generality of our approach allows us to
construct $2^{\continuum}$ closed ideals in $\cL(X,Y)$ when $X$ and
$Y$ are exotic Banach spaces such as for example $p$-convexified
Tsirelson spaces. In~\cite{bkl:20} it was shown that the projection
operators in Tsirelson and Schreier spaces generate a continuum of
distinct closed ideals. So again, an interesting distinction between
these two methods is that the operators we use to generate ideals are
finitely strictly singular whereas the projections used
in~\cite{bkl:20} are clearly not even strictly singular.

The paper is organised as follows. In the next section we give general
conditions on Banach spaces $X$ and $Y$ that ensure that $\cL(X,Y)$
contains $2^\continuum$ closed ideals. We also prove two further
results giving criteria that help with verifying those general
conditions. Each one of these two results have applications that we
present in the following two sections. In the final section we give
further remarks and state some open problems.

\section{General Conditions for having $2^{\continuum}$ closed ideals in
  $\cL(X,Y)$}
\label{sec:general-condition}

Let $X$ and $Y$ be Banach spaces and let $\cT$ be a subset of
$\cL(X,Y)$, the space of all bounded linear operators from $X$ to
$Y$. The \emph{closed ideal generated by $\cT$ }is the smallest closed
ideal in $\cL(X,Y)$ containing $\cT$ and is denoted by
$\cJ^\cT(X,Y)$. That is, $\cJ^\cT(X,Y)$ is the closure in $\cL(X,Y)$
of the set
\[
\bigg\{ \sum_{j=1}^n A_j T_j B_j:\,n\in\bn,\ (A_j)_{j=1}^n\subset
\cL(Y),\ (T_j)_{j=1}^n\subset\cT,\ (B_j)_{j=1}^n\subset\cL(X)\bigg\}
\]
consisting of finite sums of operators factoring through members of
$\cT$. When $\cT$ consists of a single operator $T\in\cL(X,Y)$, then
we write $\cJ^T(X,Y)$ instead of $\cJ^{\{T\}}(X,Y)$.

In~\cite{fsz:17}, for each $1<p<\infty$, a collection
$(T_N)_{N\subset\bn}\subset \cL(\ell_p,\co)$ of operators was
constructed such that $\cJ^{T_M}(\ell_p,\co)\neq\cJ^{T_N}(\ell_p,\co)$
whenever $M\symdif N$ is infinite. For a non-empty family $\cA$ of
subsets of $\bn$, let $\cJ_\cA$ be the closed ideal of $\cL(\ell_p,\co)$
generated by $\{T_N:\,N\in\cA\}$. There are at most a continuum of
closed ideals in $\cL(\ell_p,\co)$ that are generated by a single
operator. However, it was observed in~\cite{fsz:20} that if $\cS$ is
an almost disjoint family of cardinality~$\continuum$ consisting of
infinite subsets of $\bn$, then 
$\{\cJ_\cA:\,\cA\subset\cS,\ \cA\neq\emptyset\}$ is a lattice of
$2^{\continuum}$ distinct closed ideals in $\cL(\ell_p,\co)$.

In this section, we will present a general condition which implies
that $\cL(X,Y)$ has $2^{\continuum}$ closed ideals in the following
framework in which the above example also sits.

We are given two Banach spaces $X$ and $Y$ which are assumed to have
\emph{unconditional finite-dimensional decompositions }(UFDDs) $(E_n)$
and $(F_n)$, respectively. By this we mean that $E_n$ is a
finite-dimensional subspace of $X$ for each $n\in\bn$ and that each
element of $x$ can be written in a unique way as
$x=\sum_{n\in\bn}x_n$ with $x_n\in E_n$ for each $n\in \bn$ and that
$\sum_{n\in\bn} x_n$ converges unconditionally. We can therefore think
of the elements $x\in X$ being sequences $(x_n)$ with $x_n\in E_n$,
which we call the \emph{$n$-component }of $x$, for each $n\in\bn$.
 
As in the case of unconditional bases, this implies that for
$N\subset\bn$, the map
\[
P^X_N\colon X\to X,\quad  (x_n)_{n\in\bn}\mapsto (x_n)_{n\in N}
\]
($(x_n)_{n\in N}$ is identified with the element in $X$ whose
$m$-component vanishes for $m\in\bn\setminus N$) is well-defined and
uniformly bounded. It follows that for some $C>0$ we have 
$\bignorm{\sum_{n\in\bn}\sigma_nx_n}\leq C\bignorm{\sum_{n\in\bn}x_n}$
for all $(x_n)\in X$ and all $(\sigma_n)\in\{\pm1\}^\bn$. In this case
we say that $(E_n)$ is a \emph{$C$-unconditional finite-dimensional
  decomposition }(or \emph{$C$-unconditional FDD}) of $X$. After
renorming $X$, we can (and will) assume that $\norm{P^X_N}=1$ for a
non-empty $N\subset\bn$ and that moreover 
\begin{equation}
  \label{E:1.1}
  \Bignorm{\sum_{n\in\bn} x_n}=\Bignorm{\sum_{n\in\bn} \sigma_nx_n}
\end{equation}
for all $(x_n)\in X$ and all $(\sigma_n)\in\{\pm1\}^\bn$. We denote
for $N\subset \bn$ the image of $X$ under $P^X_N$ by $X_N$. Thus 
$X_N=P^X_N(X)=\cspn\bigcup_{n\in N}E_n$ is $1$-complemented in $X$
and $(E_n:\,n\in N)$ is a $1$-unconditional FDD of $X_N$. Similarly,
for the space $Y$ with UFDD $(F_n)$ we define $P^Y_N$ and $Y_N$ for
every $N\subset \bn$. We further assume that $\norm{P^Y_N}=1$ for
every non-empty $N\subset\bn$ and that $(F_n)$ is a $1$-unconditional
FDD of $Y$.

For each $n\in\bn$ we are given a linear operator
$T_n\colon E_n\to F_n$ and  we assume that the linear operator
\[
\ts T\colon\spn\bigcup_{n\in\bn}E_n\to\spn\bigcup_{n\in\bn}F_n, \quad
(x_n)\mapsto (T_n(x_n))
\]
extends to a bounded operator $T\colon X\to Y$. We then define for
$N\subset \bn$, the \emph{diagonal operator }$T_N\colon X_N\to Y_N$ by
$T_N=T\circ P^X_N=P^Y_N\circ T$. Note that $\norm{T_N}\leq\norm{T}$.

Our goal is to formulate conditions which imply that the following
holds for some $\Delta>0$.
\begin{equation}
  \label{E:2.1} 
  \forall M,N\in\infin{\bn}\quad\text{ if $M\setminus N\in\infin{\bn}$
    then  $\dist(T_M,\cJ^{T_N})\geq\Delta$.}
\end{equation}
Using an observation in~\cite{js:20}, we can conclude that $\cL(X,Y)$
has $2^{\continuum}$ closed ideals assuming that~\eqref{E:2.1} holds.

\begin{prop}
  \label{P:2.1}
  Let $X$, $Y$ and $(T_n)$ be as above, and assume that
  condition~\eqref{E:2.1} holds for some $\Delta>0$. Let
  $\cS\subset\infin{\bn}$ have cardinality $\continuum$ and consist of
  pairwise almost  disjoint sets. For $\cA\subset\cS$, let $\cJ_{\cA}$
  be the closed  ideal of $\cL(X,Y)$ generated by $\{T_N:\,N\in\cA\}$.
  Then if $\cA,\cB\subset\cS$ with $\cA\not=\cB$, then
  $\cJ_{\cA}\neq\cJ_{\cB}$. In particular, the cardinality of the set
  of closed ideals of $\cL(X,Y)$ is $2^{\continuum}$.
\end{prop}

\begin{proof}
  Let $\cA$ and $\cB$ be two different subsets of $\cS$. Without loss
  of generality, we assume that there is an $M\in\cA\setminus\cB$. We
  claim that $T_M\notin\cJ_{\cB}$, and that actually
  $\dist(T_M,\cJ_{\cB})\geq\Delta$.

  Indeed, let $n\in\bn$, $(A_j)_{j=1}^n\subset\cL(Y)$,
  $(B_j)_{j=1}^n\subset\cL(X)$ and $(N_j)_{j=1}^n\subset\cB$. Put
  $N=\bigcup_{j=1}^nN_j$. It follows that
  \[
  \sum_{j=1}^n A_j\circ T_{N_j}\circ B_j=  \sum_{j=1}^nA_j\circ
  P^{Y}_{N_j}\circ T_N\circ B_j\in\cJ^{T_N}\ .
  \]
  Since $M\setminus N$ is infinite, it follows from~\eqref{E:2.1} that
  \[
  \Bignorm{\sum_{j=1}^n A_j\circ T_{N_j}\circ B_j- T_M}\geq\Delta\ .
  \]
  Since $\cJ_{\cB}$ is the closure of the set of operators of the form 
  $\sum_{j=1}^n A_j\circ T_{N_j}\circ B_j$ with $n\in\bn$,
  $(A_j)_{j=1}^n\subset\cL(Y)$, $(B_j)_{j=1}^n\subset\cL(X)$ and
  $(N_j)_{j=1}^n\subset\cB$, we deduce our claim.
 \end{proof} 

In order to  separate $T_M$ from $\cJ^{T_N}$ if $M\setminus N$ is
infinite, the following condition  is sufficient.

\begin{subequations}
  \label{E:2.2} 
  \begin{align}
    &\text{\hspace{-2em}For each $n\in\bn$ there exist $l_n\in\bn$ and
      vectors $\big(x_{n,j}\big)_{j=1}^{l_n}\subset S_{E_n}$,}\notag\\
    &\text{\hspace{-2em}$\big(y^*_{n,j}\big)_{j=1}^{l_n}\subset
      S_{F^*_n}$ so that}\notag\\[2ex]
    & y^*_{n,j}\big(T_n(x_{n,j})\big)\geq 1\quad\text{for $n\in\bn$
      and $j=1,2,\dots,l_n$},\\
    & \lim_{%
      \begin{subarray}{c}
        m\to\infty\\[0.5ex]
        m\in M\setminus N
      \end{subarray}}\frac{1}{l_m} \sum_{i=1}^{l_m}
    \bignorm{T_N\circ B(x_{m,i})}=0\\
    &\qquad\text{whenever $M, N\in\infin{\bn}$ satisfy $M\setminus
      N\in\infin{\bn}$, and $B\in\cL(X)$.}\notag
  \end{align}
\end{subequations}
Indeed, for $n\in\bn$ we define the following functional
$\Psi_n\in\cL(X,Y)^*$ by
\[
\Psi_n(S)= \frac{1}{l_n}\sum_{j=1}^{l_n}
y^*_{n,j}\big(S(x_{n,j})\big),\quad\text{for $S\in\cL(X,Y)$\ .}
\]
Given $M,N\in\infin{\bn}$ with $M\setminus N\in\infin{\bn}$, we let
$\Psi$ be a $w^*$-accumulation point of
$(\Psi_m:\,m\in M\setminus N)$. It follows from~(\ref{E:2.2}a) that
\[
\Psi(T_M)\geq\liminf_{m\in M\setminus N} \Psi_m(T_M)\geq 1
\]
and for any $A\in\cL(Y)$ and $B \in\cL(X)$ it follows
from~(\ref{E:2.2}b) that
\begin{align*}
  \bigabs{\Psi(AT_N B)} &\leq\limsup_{m\in M\setminus N}
  \Bigabs{\frac1{l_m}\sum_{i=1}^{l_m} y^*_{m,i}
    \big(AT_NB(x_{m,i})\big)}\\
  &=\limsup_{m\in M\setminus N}
  \Bigabs{\frac1{l_m}\sum_{i=1}^{l_m}A^*y^*_{m,i}\big(T_NB(x_{m,i})\big)}\\
  &\leq\norm{A}\limsup_{m\in M\setminus N}
  \frac1{l_m}\sum_{i=1}^{l_m}\norm{T_NB(x_{m,i})}=0\ .
\end{align*}
Since $\norm{\Psi_n}\leq 1$ for all $n\in\bn$, it follows that
$\norm{\Psi}\leq 1$, which in turn implies condition~\eqref{E:2.1}
with $\Delta=1$.

\begin{rem}
  Some extension of the above result is possible. Assume for example
  that~\eqref{E:2.2} holds and that $U$ is an isomorphism of $Y$ into
  another Banach space $Z$. Then $\cL(X,Z)$ also has at least
  $2^\continuum$ distinct closed ideals. Indeed, by Hahn--Banach,
  there are functionals $z^*_{n,j}\in Z^*$ such that
  $U^*\big(z^*_{n,j}\big)=y^*_{n,j}$ for all $n\in\bn$ and
  $j=1,2,\dots,l_n$, and moreover,
  $C=\sup_{n,j}\norm{z^*_{n,j}}<\infty$. If we now define
  $\Psi_n\in\cL(X,Z)^*$ as above but replacing $y^*_{n,j}$ with
  $z^*_{n,j}$, then the above argument will show that
  condition~\eqref{E:2.1} holds with $\Delta=1/C$ if we replace $T_N$
  with $U\circ T_N$ for every $N\subset\bn$.
\end{rem}
  
We now want to formulate conditions on the spaces $X$ and $Y$ and the
operators $T_n\colon E_n\to F_n$, $n\in\bn$, which imply that
condition~\eqref{E:2.2} is satisfied. From now on we assume that for
each $n\in\bn$, the spaces $E_n$ and $F_n$ have $1$-unconditional,
normalized bases $(e_{n,j})_{j=1}^{\dim(E_n)}$ and
$(f_{n,j})_{j=1}^{\dim(F_n)}$ with coordinate functionals
$(e^*_{n,j})_{j=1}^{\dim(E_n)}$ and $(f^*_{n,j})_{j=1}^{\dim(F_n)} $,
respectively. \label{page:unc-basis-set-up}

We write for $n\in\bn$ the operator $T_n\colon E_n\to F_n$ as
\[
T_n\colon E_n\to F_n,\quad T_n(x)=\sum_{j=1}^{\dim(F_n)} x^*_{n,j} (x)
f_{n,j}\ ,
\]
where $x^*_{n,j}\in E_n^*$ for $n\in\bn$ and
$1\leq j\leq\dim(F_n)$. In applications, we will define the $T_n$ by
choosing the $x^*_{n,j}$ so that
\begin{equation}
  \label{E:2.2a} 
  \text{the operator }T\colon X\to Y, \quad(x_n)\mapsto (T(x_n)),
  \text{ is well defined and bounded.}
\end{equation} 
We secondly demand that $\dim(F_n)=l_n$ and that $y^*_{n,j}=f^*_{n,j}$
for $n\in\bn$ and $j=1,2,\dots,l_n$. Thus, in order to
obtain~(\ref{E:2.2}a), we require
\begin{equation}
  \label{E:2.2b}
  x^*_{n,j}(x_{n,j})\geq 1\quad \text{for all $n\in \bn$ and
    $j=1,2,\dots,l_n$.}
\end{equation}
Finally, in order to satisfy~(\ref{E:2.2}b) we will ensure that for
$m\in\bn$ and any operator $B\in\cL(E_m,X_{\bn \setminus\{m\}})$ with
$\norm{B}\leq 1$, it follows that 
\begin{equation}
  \label{E:2.2c}
  \frac1{l_m} \sum_{i=1}^{l_m} \bignorm{T_{N\setminus\{m\}}B(x_{m,i})}
  \leq\vare_m\ ,
\end{equation}
where $(\vare_m)$ is a sequence in $(0,1)$ decreasing to $0$ not
depending on $B$. Now $B$ can be written as the sum
$B=B^{(1)}+B^{(2)}$, where $B^{(1)}\in\cL(E_m,X_{\{1,2,\dots,m-1\}})$
and $B^{(2)}\in\cL(E_m,X_{\bn\setminus\{1,2,\dots,m\}})$.

To force that~\eqref{E:2.2c} holds for $B^{(1)}$ with $\vare_m/2$
instead of $\vare_m$, is not very hard: it will be  enough to ensure
that $l_m$ is very large compared to $\dim(X_{\{1,2,\dots, m-1\}})$
and that (see the proof of Proposition~\ref{P:2.2.alt} below)
$\frac1{l_m}\sup_{\pm}\bignorm{\sum_{i=1}^{l_m} \pm x_{m,i}}$
decreases to $0$ for increasing $m$. To also ensure the necessary
estimates for $B^{(2)}$, we will assume the following slightly
stronger condition.
\begin{align}
  \label{E:2.2d}
  &\lim_{m\to\infty}l_m=\infty\quad\text{and}\quad
  \lim_{l\to\infty}\sup_{m\in\bn,\ l_m\geq l}\frac{\varf_m(l)}l=0,
  \text{ where }\\
  & \varf_m(l)=\sup\Big\{ \Bignorm{\sum_{i\in A} \sigma_i x_{m,i}}:\,
  A\subset\{1,\dots,\ell_m\},\ \abs{A}\leq l,\ (\sigma_i)_{i\in A}
  \subset\{\pm1\}\Big\}.\notag
\end{align}
To ensure that~\eqref{E:2.2c} holds for $B^{(2)}$ is more complicated
and will be done in two steps. The second one of these two steps is
more straighforward: it will be enough to assume that
$T_{\bn\setminus\{1,2,\ldots m\}}$ maps vectors with small coordinates
into vectors with small norm (see condition~(a) in
Proposition~\ref{P:2.2.alt} for the precise statement). The first step
is then to assume (see condition~(b) in Proposition~\ref{P:2.2.alt})
that the following set
\[
\left\{(n,j):\,n>m,\ 1\leq j\leq l_n,\ \bigabs{x^*_{n,j}(
  B^{(2)}x_{m,i})}>\delta\textrm{ for some }1\leq i\leq l_m\right\}
\]
has small cardinality compared to $l_m$. In many situations,
guaranteeing that this set has small cardinality relative to $l_m$ is
the trickiest part, as $B^{(2)}$ is an arbitrary norm-one
operator. However, in Lemmas~\ref{L:2.3} and~\ref{L:RIP}  we present
conditions which imply this result and are stated in terms of only
basic properties of the sequences $(x_{n,j})$ and $(x^*_{n,j})$ as
well as the Banach spaces $X$ and $Y$.

Of course, since for any $N\in\infin{\bn}$, $X_N$ and $Y_N$ are
complemented subspaces of $X$ and $Y$, respectively, we can pass to
subsequences  $(E_{k_n})$, $(F_{k_n})$ and $(T_{k_n})$ for which we
are able to verify~\eqref{E:2.1}, in order to conclude that the
lattice of closed ideals of $\cL(X,Y)$ is of cardinality
$2^{\continuum}$. This follows from the following observation whose
verification is routine. Suppose that $V$ and $W$ are complemented
subspaces of $X$ and $Y$, respectively. For a closed ideal $\cJ$ in
$\cL(V,W)$, let $\cJt$ be the closure in $\cL(X,Y)$ of the set of
operators of the form $\sum_{j=1}^nA_jS_jB_j$, where $n\in\bn$,
$(A_j)_{j=1}^n\subset\cL(W,Y)$, $(S_j)_{j=1}^n\subset\cJ$ and
$(B_j)_{j=1}^n\subset\cL(X,V)$. Then $\cJt$ is a closed ideal in
$\cL(X,Y)$ and the map $\cJ\mapsto\cJt$ is injective.

\begin{prop}
  \label{P:2.2.alt}
  Assume that the spaces $X$ and $Y$, their $1$-unconditional FDDs
  $(E_n)$ and $(F_n)$ and the operators $T_n\colon E_n\to F_n $,
  $n\in\bn$, satisfy conditions~\eqref{E:2.2a},~\eqref{E:2.2b}
  and~\eqref{E:2.2d}. Assume, moreover, that the following conditions
  hold.
  \begin{mylista}{(m)}
  \item
    \label{E:2.3d.alt}
    For all $\vare>0$ and all $M\in\infin{\bn}$ there is a $\delta>0$
    and $N\in\infin{M}$ so that
    \[
    \forall x\in B_{X_N} \text{ if
      $\sup_{n\in N, 1\leq j\leq l_n}\abs{x^*_{n,j}(x)}\leq\delta$,
      then $\norm{T_N(x)}<\vare$.}
    \]
  \item
    \label{E:2.3e.alt}
    For all $\delta,\vare>0$ and all $M\in\infin{\bn}$ there are
    $m\in M$ and $N\in\infin{M}$ so that for every
    $B\in\cL(E_{m},X_N)$ with $\norm{B}\leq 1$ we have that
    \[
    \sabs{\left\{(n,j):\,n\in N,\ 1\leq j\leq l_n,\ \abs{x^*_{n,j}(
        Bx_{m,i})}>\delta\textrm{ for some }1\leq i\leq
      l_m\right\}}<\vare l_m\ .
    \]
  \end{mylista}
  Then there is a subsequence $(k_n)$ of $\bn$ so that for
  $\Et_n=E_{k_n}$, $\Ft_n=F_{k_n}$, $\Tt_n=T_{k_n}$, $\lt_n=l_{k_n}$,
  $(\xt_{n,j})_{j=1}^{\lt_n}=(x_{k_n,j})_{j=1}^{l_{k_n}}\subset\Et_n$
  and  
  $(\yt^*_{n,j})_{j=1}^{\lt_n}=(f^*_{k_n,j})_{j=1}^{l_{k_n}}\subset\big(\Ft_n\big)^*$,
  condition~\eqref{E:2.2} is satisfied. Hence, $\cL(X,Y)$ contains
  $2^\continuum$ closed ideals.
\end{prop}

\begin{proof}
  Let $(\vare_r)_{r=1}^\infty\subset (0,1)$ be a sequence which
  decreases to $0$. Put $k_0=0$ and $M_0=\bn$. We will inductively
  choose $k_r\in\bn$ and $M_r\in\infin{\bn}$ so that for all $r\in\bn$
  \begin{align}
    \label{E:2.2.a.alt}
    &\min(M_r)>k_r,\\
    \label{E:2.2.b.alt}
    &k_{r-1}<k_r,\ M_r\subset M_{r-1}\text{ and }k_r\in M_{r-1},\\
    \label{E:2.2.c.alt}
    &\frac1{l_{k_r}} \sum_{i=1}^{l_{k_r}} \norm{B(x_{k_r,i})}\leq\vare_r
    \text{ for all
    }B\in\cL(E_{k_r},X_{\{k_1,k_2,\dots,k_{r-1}\}}),\ \norm{B}\leq
    1\ ,\\
    \label{E:2.2.d.alt}
    &\frac1{l_{k_r}}\sum_{i=1}^{l_{k_r}}\norm{T_{M_r}B(x_{k_r,i})}\leq\vare_r
    \text{ for all }B\in\cL(E_{k_r},X_{M_r}),\ \norm{B}\leq 1\ .
  \end{align}
  Assume that for some $r\in\bn$, we have already chosen suitable
  $k_1<k_2<\dots< k_{r-1}$ and
  $\bn=M_0\supset M_1\supset\dots\supset M_{r-1}$.
  Put $C=\norm{T}$. By using~\ref{E:2.3d.alt}, we choose
  $\delta>0$ and $M\in\infin{M_{r-1}}$ so that
  \begin{equation}
    \label{E:2.2.4.alt}
    \norm{T_M(x)}\leq\frac{\vare_r}{2}\quad\text{ for all
      $x\in B_{X_M}$ with
      $\sup_{m\in M, 1\leq i\leq l_m}\abs{x^*_{m,i}(x)}\leq\delta$,}
  \end{equation}
  Note  that~\eqref{E:2.2.4.alt} still holds if we replace $M$ by any
  infinite subset of $M$.

  We now let $p\in\bn$ be large enough so that there exists a
  sequence $(z^*_j)_{j=1}^p\subset S_{X_{\{k_1,k_2,\dots,k_{r-1}\}}}$
  which normalizes the elements of $X_{\{k_1,k_2,\dots,k_{r-1}\}}$ up
  to the factor $2$, \ie
  \begin{equation}
    \label{E:2.2.1.alt}
    \norm{x}\leq\max_{1\leq j\leq p} 2\abs{z^*_j(x)}\quad
    \text{ for all $x\in X_{\{k_1,k_2,\dots,k_{r-1}\}}$.}
  \end{equation}
  We now apply~\eqref{E:2.2d} and choose $l\in\bn$ and $m_1>k_{r-1}$
  large enough so that for all $m\geq m_1$ we have $l_m\geq l$ and if
  $A\subset\{1,2,\dots,l_m\}$ has $\abs{A}\geq l$, then
  \begin{equation}
    \label{E:2.2.2.alt}
    \sup_\pm \Bignorm{\sum_{i\in A}\pm x_{m,i}}
    <\min\left(\frac{\delta}{C},\frac{\vare_r}{2p}\right)
    \abs{A}\ .
  \end{equation}
  For any $m\geq m_1$ and any
  $B\in\cL(E_m,X_{\{k_1,k_2,\dots,k_{r-1}\}})$ with $\norm{B}\leq 1$
  it follows that
  \begin{align}
    \label{E:2.2.3.alt}
    \frac1{l_m}\sum_{i=1}^{l_m}\bignorm{B(x_{m,i})}
    &\leq\frac2{l_m}\sum_{i=1}^{l_m}\sum_{j=1}^p\bigabs{z^*_jB(x_{m,i})}\\
    &= \frac2{l_m}\sum_{j=1}^p z^*_j\circ B
    \Big(\sum_{i=1}^{l_m}\sigma_{i,j}x_{m,i}\Big)\notag\\
    &\hspace{-2em}\big(\text{with
      $\sigma_{i,j}=\text{sign}\big(z^*_jB(x_{m,i})\big)$ for
      $1\leq i\leq l_m$ and $1\leq j\leq p$}\big)\notag\\
    &\leq\frac{2p}{l_m}\sup_\pm
    \Bignorm{\sum_{i=1}^{l_m}\pm x_{m,i}}\leq\vare_r\ . \notag
  \end{align}
  Thus~\eqref{E:2.2.c.alt} will hold for any $k_r\geq m_1$. We now
  apply assumption  \ref{E:2.3e.alt} and choose $k_r\in M$ and an
  infinite subset $M_r$ of $M$ with $m_1\leq k_r<\min(M_r)$ so that
  for every $B\in\cL(E_{k_r},X_{M_r})$ with $\norm{B}\leq 1$ we have
  that
  \begin{multline}
    \label{E:size}
    \abs{J(B)}<\frac{\vare_rl_{k_r}}{2Cl}\quad\text{where }\\
    J(B)=\left\{(n,j):\,n\in M_r,\ 1\leq j\leq
    l_n,\ \abs{x^*_{n,j}(Bx_{k_r,i})}>\delta\textrm{ for some }1\leq
    i\leq l_{k_r}\right\}.
  \end{multline}
  We now verify~\eqref{E:2.2.d.alt} and complete the inductive
  construction. Let $B\in\cL(E_{k_r},X_{M_r})$ with
  $\norm{B}\leq 1$ and set $J=J(B)$. For each $(n,j)\in J$ we denote
  \[
  I_{n,j}=\big\{
  i\in\{1,2,\dots,l_{k_r}\}:\,\abs{x^*_{n,j}(Bx_{k_r,i})}>\delta\big\}\ .
  \]
  We now have for each $(n,j)\in J$ that
  \begin{align*}
    C\sup_\pm\Bignorm{\sum_{i\in I_{n,j}}\pm x_{k_r,i}} &\geq
    \sup_\pm\Bignorm{\sum_{i\in I_{n,j}}\pm T_{M_r} B x_{k_r,i}}\\
    &\geq\sup_\pm\sum_{i\in I_{n,j}}\pm f^*_{n,j}(T_{M_r}Bx_{k_r,i})\\
    &\geq\abs{I_{n,j}}\delta
  \end{align*}
  where we used the fact that
  $f^*_{n,j}\circ T_{M_r}=x^*_{n,j}$. On the other hand, we have
  by~\eqref{E:2.2.2.alt} that if $\abs{I_{n,j}}\geq l$ then
  \[
  \sup_\pm\Bignorm{\sum_{i\in I_{n,j}}\pm x_{k_r,i}}<
  \delta\abs{I_{n,j}}/C\ .
  \]
  Thus, $\abs{I_{n,j}}<l$ for all $(n,j)\in J$. We now set
  $I=\bigcup_{(n,j)\in J} I_{n,j}$ and calculate
  \begin{align*}
    \sum_{i=1}^{l_{k_r}} \norm{T_{M_r}B(x_{k_r,i})}&\leq
    \sum_{i\in I}\norm{T_{M_r}B(x_{k_r,i})} +
    \sum_{i\notin I}\norm{T_{M_r}B(x_{k_r,i})}\\
    &\leq\sum_{i\in I}\norm{T_{M_r}B(x_{k_r,i})} +
    \vare_rl_{k_r}/2\qquad\text{by~\eqref{E:2.2.4.alt},}\\
    &\leq\sum_{(n,j)\in J}\sum_{i\in I_{n,j}}\norm{T_{M_r}B(x_{k_r,i})}
    +\vare_rl_{k_r}/2\\
    &\leq Cl\abs{J}+\vare_rl_{k_r}/2\qquad\text{as
      $\abs{I_{n,j}}<l$ for all $(n,j)\in J$,}\\
    &\leq\vare_rl_{k_r}\qquad\text{by~\eqref{E:size}.}
  \end{align*}
  Thus we have proven~\eqref{E:2.2.d.alt} and our induction is complete.

  We now prove that condition~\eqref{E:2.2}
  holds. Assumption~\eqref{E:2.2b} and the definition of $T_n$ imply
  that~(\ref{E:2.2}a) holds with $y^*_{n,j}=f^*_{n,j}$. To
  verify~(\ref{E:2.2}b), we consider infinite subsets $M$ and $N$ of
  $\{k_r:\,r\in\bn\}$ with $M\setminus N\in\infin{\bn}$. Let
  $B\in\cL(X)$ and let $m\in M\setminus N$. Define $r$ by $m=k_r$. Let
  $N_{<m}=\{n\in N:\,n<m\}$ and $N_{>m}=\{n\in N:\,n>m\}$. We then
  have the following.
  \begin{align*}
    \frac{1}{l_m}\sum_{i=1}^{l_m} \norm{T_N B(x_{m,i})}
    &\leq\frac{1}{l_m}\sum_{i=1}^{l_m}\norm{T_{N_{<m}} B(x_{m,i})}
    +\frac{1}{l_m}\sum_{i=1}^{l_m}\norm{T_{N_{>m}} B(x_{m,i})}\\
    &\leq\frac{1}{l_{k_r}}\sum_{i=1}^{l_{k_r}}C\norm{P_{\{k_1,\dots,k_{r-1}\}}
      B(x_{k_r,i})}
    +\frac{1}{l_{k_r}}\sum_{i=1}^{l_{k_r}}\norm{T_{M_r} B(x_{k_r,i})}\\
    &\leq
    \vare_rC\norm{B}+\vare_r\norm{B}\qquad\text{by~\eqref{E:2.2.c.alt}
      and~\eqref{E:2.2.d.alt}.} 
  \end{align*}
  Hence we have that
  \[
  \lim_{m\to\infty}\frac{1}{l_m}\sum_{i=1}^{l_m}\norm{T_N B(x_{m,i})}=0
  \]
  and~(\ref{E:2.2}b) is satisfied.
\end{proof}

As mentioned before, the key assumption in Proposition~\ref{P:2.2.alt}
is assumption~(b). We will now present conditions (Lemmas~\ref{L:2.3}
and~\ref{L:RIP} below) which imply this assumption. We will later give
applications in Sections~\ref{S:3} and~\ref{S:4}.

For a Banach space $Z$ with an unconditional basis $(f_j)$, we define
the \emph{lower fundamental function $\lambda_Z\colon\bn\to\br$ of $Z$
}by
\[
\lambda_Z(n)= \inf\Big\{ \Bignorm{\sum_{j\in A}
  f_j}:\,A\subset\bn,\ \abs{A}\geq n\Big\}\qquad (n\in\bn).
\]

\begin{lem}
  \label{L:2.3}
  We are given $\delta,\vare\in(0,1)$, $l\in\bn$ with $\vare l\geq 1$,
  Banach spaces $G$ and $Z$ and a 1-unconditional basis
  $(f_j)_{j=1}^\infty$ for $Z$ with biorthogonal functionals
  $(f^*_j)_{j=1}^\infty$. Assume that for some sequence
  $(x_i)_{i=1}^l\subset S_G$ we have
  \begin{equation}
    \label{eq:lower-fund-fn-estimate}
  \varf(l)/\lambda_Z(\intp{\vare l})<\delta
  \end{equation}
  where
  $\varf(l)=\sup\big\{\norm{\sum_{i\in{}I}\sigma_ix_i}:\,I\subset\{1,2,\dots,l\},\ (\sigma_i)_{i\in{}I}\subset\{\pm{}1\}\big\}$.
  Then for any $B\colon G\to Z$ with $\norm{B}\leq 1$ we have
  \[
  \bigabs{\big\{j\in\bn:\,\abs{f^*_j(Bx_i)}>\delta\text{ for some
    }1\leq i\leq l\big\}}\leq \vare l\ .
  \]
\end{lem}

\begin{proof}
  Fix an operator $B\colon G\to Z$ with $\norm{B}\leq 1$ and set
  \begin{align*}
    I&=\big\{i\in\{1,2,\dots,l\}:\,\abs{f^*_j(Bx_i)}>\delta\text{ for
      some }j\in\bn\big\}\ ,\\
    J&=\big\{j\in\bn:\,\abs{f^*_j(Bx_i)}>\delta\text{ for some }1\leq
    i\leq l\big\}\ .
  \end{align*}
  We next fix independent Rademacher random variables
  $(r_i)_{i\in I}$ and establish the estimate
  \begin{equation}
    \label{eq:lem:ave-of-coords}
    \be \biggabs{\sum_{i\in I} r_i f^*_j\big(B(x_i)\big)}
    >\delta \qquad \text{for all }j\in J\ .
  \end{equation}
  To see this, fix $j\in J$ and set $y_i=f^*_j\big(B(x_i)\big)$ for 
  $i\in I$. By the definition of $J$, there is an $i_0\in I$ such
  that $\abs{y_{i_0}}>\delta$. Thus, by Jensen's inequality we have
  \begin{eqnarray*}
    \be \biggabs{\sum_{i\in I} r_i y_i} &=& \be \biggabs{\sum_{i\in I}
      r_{i_0} r_i y_i} = \be \biggabs{y_{i_0} + \sum _{i\in I,\ i\neq
        i_0} r_{i_0}r_i y_i} \\[2ex]
    & \geq & \biggabs{y_{i_0}  + \sum _{i\in I,\ i\neq i_0} \be
      (r_{i_0}r_i) y_i} = \abs{y_{i_0}}>\delta\ .
  \end{eqnarray*}
  We then calculate
  \begin{eqnarray*}
    \varf (l) &\geq & \be \biggnorm{\sum_{i\in I} r_i
      B(x_i) }_Z \qquad \text{as $\norm{B}\leq
      1$,}\\[2ex]
    & = & \be \biggnorm{\sum_j \biggabs{\sum_{i\in I} r_i
        f^*_j\big(B(x_i)\big)} f_j}_Z \quad\text{as $(f_j)$ is
      1-unconditional,}\\[2ex]
    &\geq & \biggnorm{\sum_j \be \biggabs{\sum_{i\in I} r_i
        f^*_j\big(B(x_i)\big)} f_j}_Z \qquad \text{by
      Jensen's inequality,}\\[2ex]
    &\geq & \delta \biggnorm{\sum_{j\in J}f_j}_Z
    \quad\text{using \eqref{eq:lem:ave-of-coords} and the
      1-unconditionality of $(f_j)$ ,}\\[2ex]
    &\geq & \delta \lambda_Z(\abs{J})\ .
  \end{eqnarray*}
  Since the lower fundamental function $\lambda_Z$ is clearly
  increasing, it follows from
  assumption~\eqref{eq:lower-fund-fn-estimate} that
  $\abs{J}\leq\vare l$.
\end{proof}

We now state and prove a very different condition that also implies
assumption~(b) in Proposition~\ref{P:2.2.alt}. Here we use the
notation and framework established on
page~\pageref{page:unc-basis-set-up}.

\begin{lem}
  \label{L:RIP}
  Let $1\leq s,t<\infty$ and suppose the following holds.
  \begin{mylista}{(m)}
  \item
    There is a constant $c_1>0$ so that
    $(e^*_{m,i})_{i=1}^{\dim(E_m)}$ is $c_1$-dominated by the unit
    vector basis of $\ell_s$ for each $m\in\bn$. That is,
    \[
    \snorm{\sum_{i=1}^{\dim(E_m)}a_i e_{m,i}^*}\leq
    c_1\left(\sum_{i=1}^{\dim(E_m)}\abs{a_i}^s\right)^{1/s}\qquad\text{for
      all scalars } (a_i)_{i=1}^{\dim(E_m)}\ ,
    \]
  \item
    There is a constant $c_2>0$ so that for all $m,n\in\bn$ with $m<n$
    and all $A\subset\{1,2,\dots,l_n\}$ with $\abs{A}\leq l_m$, the
    sequence $(x^*_{n,j})_{j\in A}$ is $c_2$-weak $\ell_s$. That is, 
    \[
    \bigg(\sum_{j\in A}\abs{x^*_{n,j}(x)}^s\bigg)^{1/s}\leq
    c_2\norm{x}\qquad\text{for all }x\in E_n\ .
    \]
  \item
    There is a constant $c_3>0$ so that if $z_n\in S_{E_n}$ for all
    $n\in\bn$ then $(z_n)_{n=1}^\infty$ $c_3$-dominates the unit
    vector basis for $\ell_t$.  In other words,
    \[
    \bigg(\sum_{n\in\bn} \norm{P^X_n x}^t\bigg)^{1/t}\leq c_3
    \norm{x}\qquad\text{for all }x\in X\ .
    \]
  \item
    $\lim_{m\to\infty} \big(\dim(E_m)\big)^{max(1,t/s)}l_m^{-1}=0$.
  \end{mylista}
  Then for all $\delta,\vare>0$ there exists $m\in\bn$ so that for all
  $N\in\infin{\{n\in\bn:\,n\geq m+1\}}$ and for all $B\in\cL(E_m,X_N)$
  with $\norm{B}\leq 1$, the set
  \[
  J=\left\{(n,j):\,n\in N,\ 1\leq j\leq
  l_n,\ \abs{x^*_{n,j}(Bx_{m,i})}>\delta\textrm{ for some }1\leq i\leq
  l_m\right\}
  \]
  has $\abs{J}\leq\vare l_m$.
\end{lem}
\begin{proof}
  Let $0<\delta,\vare<1$, 
  $m\in\bn$, $N\in\infin{\{n\in\bn:\,n\geq m+1\}}$ and
  $B\in\cL(E_m,X_N)$ with $\norm{B}\leq 1$. Let $H\subset J$ be such
  that $\abs{H}\leq l_m$. Note that if we prove that
  $\abs{H}<\vare l_m$ then we have that $\abs{J}<\vare l_m$. For each
  $n\in N$ denote $H_n=\big\{j\in\{1,2,\dots,l_n\}:\,(n,j)\in H\}$. We
  have that
  \begin{align*}
    \delta^s\abs{H_n}&\leq \sum_{j\in H_n} \norm{B^* x_{n,j}^*}^s\\
    &=\sum_{j\in H_n}\Bignorm{\sum_{i=1}^{\dim(E_m)}
      \big(B^*x^*_{n,j}(e_{m,i})\big)e^*_{m,i}}^s\\
    &\leq c_1^s \sum_{j\in H_n} \sum_{i=1}^{\dim(E_m)}
    \abs{B^*x^*_{n,j}(e_{m,i})}^s\qquad\text{by (a),}\\
    &= c_1^s  \sum_{i=1}^{\dim(E_m)}\sum_{j\in H_n} \abs{x^*_{n,j}(P^X_n
      B e_{m,i})}^s\\
    &\leq  c_1^s c_2^s \sum_{i=1}^{\dim(E_m)} \norm{P^X_n B
      e_{m,i}}^s\qquad\text{by~(b).}
  \end{align*}
  For the case that $t\leq s$, we may use the fact that
  $\norm{P^X_nBe_{m,i}}\leq 1$ to obtain
  \begin{equation}
    \label{E:ts}
    \delta^{s}\abs{H_n}\leq c_1^s c_2^s \sum_{i=1}^{\dim(E_m)}
    \norm{P^X_nBe_{m,i}}^s\leq c_1^s c_2^s \sum_{i=1}^{\dim(E_m)}
    \norm{P^X_nBe_{m,i}}^t\ .
  \end{equation}  
  For the case that $s<t$, H\"olders inequality gives that
  \[
  \delta^{s}\abs{H_n}\leq  c_1^s c_2^s \sum_{i=1}^{\dim(E_m)}
  \norm{P^X_nBe_{m,i}}^s\leq c_1^s c_2^s
  \big(\dim(E_m)\big)^{\frac{t-s}t}\bigg(\sum_{i=1}^{\dim(E_m)}
  \norm{P_nBe_{m,i}}^t\bigg)^{s/t}\ .
  \]
  By raising the above inequality to the power $t/s$, we have for
  $s<t$ that
  \begin{equation}
    \label{E:st}
    \delta^{t}\abs{H_n}\leq\delta^{t}\abs{H_n}^{t/s} \leq c_1^tc_2^t
    \big(\dim(E_m)\big)^{\frac{t-s}s} \sum_{i=1}^{\dim(E_m)}
    \norm{P^X_nBe_{m,i}}^t\ .
  \end{equation}
  We now finish the proof for the case that $t\leq s$, and we will
  consider the remaining case later. Summing~\eqref{E:ts} over
  $n\in N$ gives that
  \begin{align*}
    \abs{H}&=\sum_{n\in N}\abs{H_n}\\
    &\leq \delta^{-s}c_1^s c_2^s \sum_{i=1}^{\dim(E_m)}\sum_{n\in N}
    \norm{P^X_nBe_{m,i}}_{E_n}^t\\
    &\leq \delta^{-s}c_1^s c_2^s \sum_{i=1}^{\dim(E_m)}c_3^t
    \norm{Be_{m,i}}^t\qquad\text{by (c),}\\
    &\leq \delta^{-s}c_1^s c_2^s c_3^t \dim(E_m)\ .
  \end{align*}
  As $t\leq s$ we have by~(d) that
  $\lim_{m\to\infty}\dim(E_m)l_m^{-1}=0$. Hence, if $m\in\bn$ is large
  enough then $\abs{H}<\vare l_m$, and thus $\abs{J}<\vare l_m$ as
  well.
    
  We now consider the remaining case that $s<t$. By~\eqref{E:st} we
  have that
  \begin{align*}
    \abs{H}&=\sum_{n\in N}\abs{H_n}\\
    &\leq\delta^{-t}c_1^{t} c_2^{t} \big(\dim(E_m)\big)^{\frac{t-s}s}
    \sum_{i=1}^{\dim(E_m)} \sum_{n\in N} \norm{P^X_nBe_{m,i}}_{E_n}^t\\
    &= \delta^{-t}c_1^t c_2^t c_3^t \big(\dim(E_m)\big)^{\frac{t-s}s}
    \sum_{i=1}^{\dim(E_m)} \norm{Be_{m,i}}^t\qquad\text{by (c),}\\
    &\leq \delta^{-t}c_1^{t} c_2^{t} c_3^t
    \big(\dim(E_m)\big)^{\frac{t-s}s} \dim(E_m)\\
    &= \delta^{-t}c_1^{t} c_2^{t} c_3^t
    \big(\dim(E_m)\big)^{\frac{t}s}\ .
  \end{align*}
  As $s<t$ we have by~(d) that
  $\lim_{m\to\infty}\big(\dim(E_m)\big)^{\frac{t}s}l_m^{-1}=0$. Hence,
  if $m\in\bn$ is large enough then $\abs{H}<\vare l_m$, and thus
  $\abs{J}<\vare l_m$.
\end{proof}

\section{Applications I}
\label{S:3}
 
In this section we apply the general process developed in
Section~\ref{sec:general-condition} together with Lemma~\ref{L:2.3} to
establish a class of pairs $(X,Y)$ of Banach spaces for which
$\cL(X,Y)$ contains $2^\continuum$ distinct closed ideals. We will
then give a list of examples including classical $\ell_p$-spaces and
$p$-convexified Tsirelson spaces.

\begin{thm}
  \label{thm:ufdds-with-lower-upper-estimates}
  Let $1<p\leq r<2$ and $1<r<q<\infty$. Let $X$ be an unconditional
  sum of a sequence $(E_n)$ of finite-dimensional Banach spaces
  satisfying a lower $\ell_r$-estimate, and assume that the $E_n$
  contain uniformly complemented, uniformly isomorphic copies of
  $\ell_p^m$. Let $Y$ be an unconditional sum of a sequence $(F_n)$ of
  finite-dimensional Banach spaces satisfying an upper
  $\ell_q$-estimate. Then $\cL(X,Y)$ contains $2^\continuum$ distinct
  closed ideals.
\end{thm}

Let us first recall some of the terminology used here. To say that $X$
\emph{is an unconditional sum of a sequence $(E_n)$ }of
finite-dimensional Banach spaces means that $X$ consists of all
sequences $(x_n)$ with $x_n\in E_n$ for all $n\in\bn$, and there is an
unconditional basis $(u_n)$ of some Banach space such that the norm of
an element $(x_n)$ of $X$ is given by
\[
\bignorm{(x_n)}=\Bignorm{\sum_n \norm{x_n}u_n}\ .
\]
If the $(u_n)$ is a $C$-unconditional basis, then we say that $X$ is a
\emph{$C$-unconditional sum of $(E_n)$}. In this case $(E_n)$ is a
UFDD for $X$, but the converse is not true in general.

We say that $X$ \emph{satisfies a lower $\ell_r$-estimate }if $(u_n)$
dominates the unit vector basis of $\ell_r$, \ie if for some $c>0$ and
for all $(x_n)\in X$, the estimate
\[
\Bignorm{\sum_n x_n} \geq c\Big(\sum_n\norm{x_n}^r\Big)^{1/r}
\]
holds. In this case we say that $X$ \emph{satisfies a lower
  $\ell_r$-estimate with constant $c$}. An upper $\ell_r$-estimate
is defined analogously in the obvious way. To say that the $E_n$
contain uniformly complemented, uniformly isomorphic copies of
$\ell_p^m$ means that for some $C>0$ and for all $m\in\bn$ there
exist $n\in\bn$ and a projection $P_n\colon E_n\to E_n$ with
$\norm{P_n}\leq C$ whose image is $C$-isomorphic to $\ell_p^m$.

The special case of $X=\ell_p$ and $Y=\ell_q$ was treated
in~\cite{sz:18} where the existence of a continuum of distinct closed
ideals was established. Here we shall make use of finite-dimensional
versions of Rosenthal's $X_{p,w}$ spaces which were also the main
ingredient in~\cite{sz:18}. We begin by recalling the definition and
relevant properties.

Given $2<p<\infty$, $0<w\leq 1$ and $n\in\bn$, we denote by
$E^{(n)}_{p,w}$ the Banach space $\big(\br^n,\norm{\cdot}_{p,w}\big)$,
where 
\[
\bignorm{(a_j)_{j=1}^n}_{p,w}=\bigg( \sum_{j=1}^n
\abs{a_j}^{p}\bigg)^{\frac{1}{p}} \join w\bigg( \sum_{j=1}^n
\abs{a_j}^2\bigg)^{\frac12}\ .
\]
We write $\big\{ e^{(n)}_j:\,1\leq j\leq n\big\}$ for the unit
vector basis of $E^{(n)}_{p,w}$, and we denote by
$\big\{e^{(n)*}_j:\,1\leq j\leq n\big\}$ the unit vector basis of the
dual space $\big(E^{(n)}_{p,w}\big)^*$, which is biorthogonal to the
unit vector basis of $E^{(n)}_{p,w}$.

Given $1<p<2$, $0<w\leq 1$ and $n\in\bn$,  we fix once and for all a 
sequence $f^{(n)}_j=f^{(n)}_{p,w,j}$, $1\leq j\leq n$, of independent
symmetric, 3-valued random variables with $\norm{f^{(n)}_j}_{L_p}=1$
and $\norm{f^{(n)}_j}_{L_2}=\frac1{w}$ for $1\leq j\leq n$ (these two
equalities determine the distribution of a 3-valued symmetric random
variable). We then define $F^{(n)}_{p,w}$ to be the subspace
$\spn\big\{f^{(n)}_j:\,1\leq j\leq n\big\}$ of $L_p$. It follows from
the work of Rosenthal~\cite{rosenthal:70a} that there exists a
constant $K_p>0$ dependent only on $p$ so that for all scalars
$(a_j)_{j=1}^n$ we have
\begin{equation}
 \label{eq:xp-isom-p<2-n}
 \frac1{K_p}\biggnorm{\sum_{j=1}^n a_je^{(n)*}_j} \leq
 \biggnorm{\sum_{j=1}^n a_j f^{(n)}_j}_{L_p} \leq
 \biggnorm{\sum_{j=1}^n a_je^{(n)*}_j}\ ,
\end{equation}
where $\big\{ e^{(n)*}_j:\,1\leq j\leq n\big\}$ is the unit vector basis
of the dual space $\big(E^{(n)}_{p',w}\big)^*$ as defined above and
$p'$ is the conjugate index of $p$. Since the random variables
$f^{(n)}_j$ are 3-valued, $F^{(n)}_{p,w}$ is a subspace of the span of
indicator functions of $3^n$ pairwise disjoint sets. Thus, we can and
will think of $F^{(n)}_{p,w}$ as a subspace of $\ell_p^{3^n}$. The
following result follows directly from Rosenthal's
work~\cite{rosenthal:70a}.

\begin{prop}\cite{sz:18}*{Proposition 1}
 \label{prop:properties-of-Fn}
 Let $1<p<2$, $0<w\leq 1$ and $n\in\bn$. Then
 \begin{mylist}{(iii)}
 \item
   $\big\{f^{(n)}_{j}:\,1\leq j\leq n\big\}$ is a normalized,
   1-unconditional basis of $F^{(n)}_{p,w}$.
 \item
   There exists a projection
   $P^{(n)}_{p,w}\colon\ell_p^{3^n}\to\ell_p^{3^n}$ onto
   $F^{(n)}_{p,w}$ with $\bignorm{P^{(n)}_{p,w}}\leq K_p$.
 \item
   For each $1\leq k\leq n$ and for every $A\subset \{1,\dots,n\}$
   with $\abs{A}=k$ we have
   \[
   \frac1{K_p}\cdot \Big(k^{\frac{1}{p}} \meet \tfrac{1}{w}
   k^{\frac12}\Big)
   \leq\biggnorm{\sum_{j\in A} f^{(n)}_j} \leq k^{\frac{1}{p}} \meet
   \tfrac{1}{w} k^{\frac12}\ .
   \]
 \end{mylist}
\end{prop}

The lower estimate of the lower fundamental function in
Lemma~\ref{L:2.2} below follows easily from~\cite{sz:18}*{Lemma 3} and
its proof.

\begin{lem}
  \label{L:2.2}
  Given an increasing sequence $(k_n)$ in $\bn$ and a decreasing
  sequence $(w_n)$ in $(0,1]$, let $1<p\leq r<2$ and let
  $Z$ be a $1$-unconditional sum of $\big(F^{(k_n)}_{p,w_n}\big)$
  satisfying a lower $\ell_r$-estimate with constant~$1$. Then with
  respect to the unconditional basis
  $(f^{(k_n)}_j:\,n\in\bn,\ 1\leq j\leq k_n)$ of $Z$, for all
  $m\in\bn$ we have
  \[  
  \lambda_Z(m) \geq \frac1{K_p}\bigg(
  \Big(\frac{m}{2}\Big)^{1/r}\meet\bigg(\sum_{j=1}^{s-1}
  \frac{k_j}{w_j^2}+\frac{t}{w_{s}^2}\bigg)^{1/2}\bigg)
  \]
  where $s=s(m)\in\bn$ is maximal so that
  $\sum_{j=1}^{s-1} k_j\leq m/2$ and $t=m/2-\sum_{j=1}^{s-1} k_j$. 
  In particular, if $m\leq k_1$ then
  \[
  \lambda_Z(m) \geq \frac1{2K_p} \bigg(m^{1/r}
  \meet\frac{m^{1/2}}{w_1}\bigg)\ .
  \]
\end{lem}

Let us denote by $\big(e^{(n)}_{2,j}\big)_{j=1}^n$ the unit vector
basis of $\ell_2^n$. We will need the following lemma
from~\cite{sz:18}. Recall that $p'$ is the conjugate index of $p$.

\begin{lem}\cite{sz:18}*{Lemma 5}
 \label{lem:F_n-to-ell_2^n}
 Given $1<p<2$ and $p<q<\infty$, let $n\in\bn$, $w\in(0,1]$ and
 $F=F^{(n)}_{p,w}$. Let $y=\sum_{j=1}^n y_jf^{(n)}_j\in F$ with
 $\norm{y}_{F}\leq 1$, and let
 $\yt=\sum_{j=1}^n y_je^{(n)}_{2,j}\in\ell_2^n$. If
 $\norm{y}_\infty=\max_j\abs{y_j}\leq\sigma\leq 1$ and
 $w\leq\sigma^{\frac12-\frac1{p'}}= \sigma^{\frac1p-\frac1{2}}$, then
 \[
 \norm{\yt}_{\ell_2^n}^q \leq D\sigma^s\cdot
 \norm{y}_F^p\ ,
 \]
 where $D$ only depends on $p$ and $q$, and
 $s=\min\big\{\frac{q}2-\frac{p}2\ ,\ \frac{q}2-\frac{q}{p'}\big\}$.
\end{lem}

\begin{proof}[Proof of Theorem~\ref{thm:ufdds-with-lower-upper-estimates}]
  Choose $\eta\in(0,1)$ so that $\eta<\frac1{r}-\frac1{2}$ and for
  each $n\in\bn$ let $w_n=n^{-\eta}$.
  After passing to a complemented subspace of $X$ using
  Proposition~\ref{prop:properties-of-Fn}, and after passing to an
  equivalent norm, we may assume that $X$ is a $1$-unconditional sum
  of $(E_n)$ satisfying a lower $\ell_r$-estimate with constant~$1$,
  where $E_n=F^{(n)}_{p,w_n}$ for all $n\in\bn$. Also, using Dvoretzky's
  theorem, after passing to a subspace of $Y$ with suitable renorming,
  we may assume that $Y$ is a $1$-unconditional sum of $(F_n)$
  satisfying an upper $\ell_q$-estimate with constant~$1$, where
  $F_n=\ell_2^n$ for all $n\in\bn$ (\cf Remark following
  condition~\eqref{E:2.2}).

  We will now follow the scheme developed in
  Section~\ref{sec:general-condition}.
  For each $m\in\bn$ we let $l_m=m$, $x_{m,i}=f^{(m)}_i\in E_m$ and
  $f_{m,i}=e^{(m)}_{2,i}\in F_m$ for $1\leq i\leq m$, and define
  $T_m\colon E_m\to F_m$ by
  $T_m(x)=\sum_{i=1}^mx^*_{m,i}(x)f_{m,i}$, where $x^*_{m,i}$ are the
  biorthogonal functionals to the $1$-unconditional basis
  $(x_{m,i})_{i=1}^m$ of $E_m$. We are thus in the situation
  described in Proposition~\ref{P:2.2.alt}. It remains to verify
  assumptions~\ref{E:2.3d.alt} and~\ref{E:2.3e.alt} of the
  proposition as well as the general
  assumptions~\eqref{E:2.2a},~\eqref{E:2.2b} and~\eqref{E:2.2d}.

  Assumption~\eqref{E:2.2b} is clear. Next, it follows from
  Proposition~\ref{prop:properties-of-Fn}(i) that 
  $\sup_n\norm{T_n}$ is bounded by the cotype-2 constant of
  $L_p$. Since $r<q$, it follows from the upper $\ell_q$-estimate on
  $Y$ and the lower $\ell_r$-estimates of $X$ that~\eqref{E:2.2a}
  holds.

  Using Proposition~\ref{prop:properties-of-Fn} again, we note that
  \[
  \varf_m(l)\leq l^{\frac{1}{p}} \meet \tfrac{1}{w_m} l^{\frac12}
  \]
  for all $l\leq m$ in $\bn$, and condition~\eqref{E:2.2d} follows.

  We next turn to condition~\ref{E:2.3d.alt} of
  Proposition~\ref{P:2.2.alt}. Fix $\vare>0$ and 
  $M\in\infin{\bn}$. Choose $\delta\in(0,1)$ so that
  $\big(D\delta^s\big)^{\frac{r}{p}}<\vare^q$, where $D$ and $s$ are
  given by Lemma~\ref{lem:F_n-to-ell_2^n} with $q$ replaced by
  $\frac{pq}{r}$. Then choose $N\in\infin{M}$ so that
  $w_n\leq\delta^{\frac1p-\frac1{2}}$ for all $n\in N$. Now fix $x\in
  B_{X_N}$ with
  $\sup_{n\in N,\ 1\leq j\leq n}\abs{x^*_{n,j}(x)}\leq\delta$. Writing
  $x=\sum_{n\in N}\sum_{j=1}^n a_{n,j}x_{n,j}$, we have
  $\abs{a_{n,j}}\leq\delta$ for all $n\in N$ and $1\leq j\leq n$. It
  follows from Lemma~\ref{lem:F_n-to-ell_2^n} that
  \[
  \left(\sum_{j=1}^n\abs{a_{n,j}}^2\right)^{\frac{pq}{2r}}\leq
  D\delta^s \Bignorm{\sum_{j=1}^na_{n,j}x_{x,j}}_{E_n}^p\ ,
  \]
  and hence
  \[
  \left(\sum_{j=1}^n\abs{a_{n,j}}^2\right)^{\frac{q}{2}}\leq
  \big(D\delta^s\big)^{\frac{r}{p}}
  \Bignorm{\sum_{j=1}^na_{n,j}x_{x,j}}_{E_n}^r
  \]
  for every $n\in N$. Summing over $n\in N$ and using the lower
  $\ell_r$-estimate of $X$ and the upper $\ell_q$-estimate of $Y$, we
  obtain
  \[
  \norm{T_N(x)}_Y^q\leq
  \big(D\delta^s\big)^{\frac{r}{p}}\norm{x}_{X_N}^r<\vare^q\ ,
  \]
  which completes the proof of condition~\ref{E:2.3d.alt}.

  To verify condition~\ref{E:2.3e.alt} of Proposition~\ref{P:2.2.alt},
  we fix $\delta,\vare\in(0,1)$ and 
  $M\in\infin{\bn}$. We first choose $m\in M$ so that
  $m\vare\geq1$ and
  \[
  \frac{2K_pm^{\eta+\frac12}}{\mt^{\frac{1}{r}}}<\delta\qquad\text{where
    }\mt=\intp{\vare m}\ .
  \]
  We then choose $N\in\infin{M}$ so that $n=\min(N)$ satisfies
  $\mt^{\frac1{r}}\leq \mt^{\frac12}/w_n$. We now apply
  Lemma~\ref{L:2.3} with $G=E_m$, $l=m$ and $Z=X_N$. First note that
  by Proposition~\ref{prop:properties-of-Fn}(iii), we have
  \[
  \varf_m(m)\leq m^{\frac1{p}}\meet
  \frac{m^{\frac12}}{w_m}=m^{\eta+\frac12}\ .
  \]
  On the other hand, it follows from Lemma~\ref{L:2.2} that
  \[
  \lambda_{X_N}(\mt)\geq \frac1{2K_p} \bigg(\mt^{1/r}
  \meet\frac{\mt^{1/2}}{w_n}\bigg)=\frac{\mt^{1/r}}{2K_p}
  \]
  by the choice of $N$. Hence,
  $\varf_m(m)/\lambda_{X_N}(\intp{\vare{}m})\leq\frac{2K_pm^{\eta+\frac12}}{\mt^{\frac{1}{r}}}<\delta$
  by the choice of $m$. An application of Lemma~\ref{L:2.3} shows that
  for any $B\in\cL(E_m,X_N)$ with $\norm{B}\leq 1$ we have
  \[
  \sabs{\left\{(n,j):\,n\in N,\ 1\leq j\leq n,\ \abs{x^*_{n,j}(
      Bx_{m,i})}>\delta\textrm{ for some }1\leq i\leq m\right\}}<\vare
  m\ .
  \]
  This shows that~\ref{E:2.3e.alt} of Proposition~\ref{P:2.2.alt}
  holds and the proof of the theorem is thus complete.
\end{proof}

\begin{rem}
  It is not difficult to prove (\cf~\cite{sz:18}*{Proposition~8}) that
  the $2^\continuum$ closed ideals constructed in the proof of
  Theorem~\ref{thm:ufdds-with-lower-upper-estimates} are all contained
  in the ideal of finitely strictly singular operators.
\end{rem}

\begin{cor}
  \label{cor:upper-and-lower-estimate-examples}
  Let $1<p<q<\infty$ and let $p'$ and $q'$ denote the conjugate
  indices of $p$ and $q$, respectively. Let $X$ be one of the spaces
  $\ell_p$, $T_p$ or $T_{p'}^*$. Let $Y$ be one of the spaces
  $\ell_q$, $T_q$ or $T_{q'}^*$. Then $\cL(X,Y)$ has exactly
  $2^\continuum$ closed ideals. It follows that $\cL(X\oplus Y)$ also
  has exactly $2^\continuum$ closed ideals.
\end{cor}
\begin{proof}
  We recall the following properties of the $p$-convexified Tsirelson
  space $T_p$ which can be found in~\cite{cas-shur:89}. Its unit
  vector basis $(t_n)$ is normalized, $1$-unconditional, dominated by
  the unit vector basis of $\ell_p$ and dominates the unit vector
  basis of $\ell_r$ whenever $p<r<\infty$. Moreover, given a sequence
  $(I_n)$ of consecutive intervals of positive integers with
  $1\in I_1$, if we let $E_n=\spn\{t_i:\,i\in I_n\}$ and pick any
  $k_n\in I_n$ for every $n\in\bn$, then $T_p$ is isomorphic to the
  unconditional sum of $(E_n)$ with respect to the unconditional basis
  $(t_{k_n})$. It follows from
  Theorem~\ref{thm:ufdds-with-lower-upper-estimates} that $\cL(X,Y)$
  has exactly $2^\continuum$ closed ideals when $1<p<2$, and the same
  then holds by duality when $2\leq p<\infty$.

  It follows by standard basis techniques that every operator from $Y$
  to $X$ is compact. Hence non-trivial closed ideals of $\cL(X,Y)$
  correspond to non-trivial closed ideals of $\cL(X\oplus Y)$ as
  follows. We think of operators on $X\oplus Y$ as $2\times 2$
  matrices in the obvious way. Given a non-trivial closed ideal $\cJ$
  in $\cL(X,Y)$, it is easy to see that
  \[
  \cJt=\left\{%
  \begin{pmatrix}
    A & B\\
    C & D
  \end{pmatrix}:\, A\in\cK(X),\ B\in\cL(Y,X),\ C\in\cJ,\ D\in\cK(Y)
  \right\}
  \]
  is a closed ideal of $\cL(X\oplus Y)$, and moreover, the map
  $\cJ\mapsto\cJt$ is injective. It follows that $\cL(X\oplus Y)$ also
  has $2^\continuum$ closed ideals, and this completes the proof of
  the theorem.
\end{proof}

As mentioned in the Introduction, the above result implies the recent
result of Johnson and Schechtman~\cite{js:20} that $\cL(L_p)$ contains
$2^\continuum$ closed ideals for $1<p\neq 2<\infty$.

\begin{cor}
  Let $1<p\neq 2<\infty$. The algebra $\cL(L_p)$ of operators on $L_p$
  contains exactly $2^\continuum$ closed ideals.
\end{cor}
  
\section{Applications II}
\label{S:4}
    
As in the previous section, we will apply the general process
developed in Section~\ref{sec:general-condition} to establish a class
of pairs $(X,Y)$ of Banach spaces for which $\cL(X,Y)$ contains
$2^\continuum$ distinct closed ideals. However, we will be using
Lemma~\ref{L:RIP} in this section as opposed to Lemma~\ref{L:2.3}.

Let $1\leq p<q\leq\infty$. Suppose that
$\big(\ell_2^{n}\big)_{n=1}^\infty$ is a UFDD for a Banach space $X$
with a lower $\ell_p$-estimate and that
$\big(\ell_\infty^{n}\big)_{n=1}^\infty$ is a UFDD for a Banach space
$Y$ with an upper $\ell_q$-estimate. We will prove that $\cL(X,Y)$
contains $2^\continuum$ distinct closed ideals. As
$\big(\bigoplus_{n=1}^\infty\ell_2^{n}\big)_{\ell_p}$ is complemented
in $\ell_p$ for all $1<p<\infty$, we obtain that $\cL(\ell_p,\co)$
contains $2^{\continuum}$ distinct closed ideals for all $1<p<\infty$,
which proves that our general setup incorporates the results presented 
in~\cite{fsz:20}. By duality, we obtain  that $\cL(\ell_1,\ell_p)$ and
$\cL(\ell_p,\ell_\infty)$ each contain $2^{\continuum}$ distinct
closed ideals. Hence, the cardinality of the set of closed ideals is
exactly $2^\continuum$ for each of $\cL(\ell_p\oplus\co)$,
$\cL(\ell_p\oplus\ell_\infty)$ and $\cL(\ell_1\oplus\ell_p)$ for all
$1<p<\infty$. Note that we also obtain that the cardinality of the set
of closed ideals in
$\cL\big(\big(\bigoplus_{n=1}^\infty\ell_2^{n}\big)_{\ell_1}\oplus\co\big)$
is $2^\continuum$, however we are not able to conclude anything about
$\cL({\ell_1}\oplus \co)$ as the finite-dimensional spaces $\ell_2^n$
are not uniformly complemented in $\ell_1$.

In the previous section, for each $n\in\bn$, the operator
$T_n\colon E_n\to\ell_2^n$ was the formal identity between two
$n$-dimensional Banach spaces. Now, we will choose sequences
$k_1<l_1<k_2<l_2<\dots$ and operators
$T_n\colon\ell_2^{k_n}\to\ell_\infty^{l_n}$. When considered as a
matrix, each $T_n$ will be much taller than it is wide.  

Let $1\leq p<\infty$.  The probabilistic proofs for the existence of
RIP (Restricted Isometry Property) matrices from compressed
sensing~\cite{foucart-rauhut:13} show that there exist sequences
$k_1<l_1<k_2<l_2<\dots$ with
$\lim_{n\to\infty} k_n^{\max(1,p/2)} l_n^{-1}=0$ such that if unit
vectors $(x_{n,j})_{j=1}^{l_n}$ are randomly chosen with uniform
distribution in $\ell_2^{k_n}$ then with high probability we have for
all $J\subset\{1,2,\dots,l_n\}$ with $\abs{J}\leq l_{n-1}$ that
\begin{equation}
  \label{E:RIP1}
  \frac{1}{2}\sum_{j\in J}\abs{a_j}^2\leq\Bignorm{\sum_{j\in
      J}a_jx_{n,j}}^2\leq 2\sum_{j\in J}\abs{a_j}^2\textrm{ for all
  }(a_j)_{j\in J}\subset\br\ ,
\end{equation}
\begin{equation}
  \label{E:RIP2}
  \sum_{j\in J}\abs{\ip{x}{x_{n,j}}}^2\leq 2\norm{x}^2\textrm{ for all
  }x\in\ell_2^{k_n}\ .
\end{equation}
We now show how this
construction satisfies the conditions of Proposition~\ref{P:2.2.alt}
and Lemma~\ref{L:RIP}. 

\begin{thm}
  \label{T:UP}
  Let $1\leq p<q\leq\infty$.  Suppose that $(\ell_2^{n})_{n=1}^\infty$
  is a UFDD for $X$ with a lower $\ell_p$-estimate and that
  $(\ell_\infty^{n})_{n=1}^\infty$ is a UFDD for $Y$ with an upper
  $\ell_q$-estimate. Then $\cL(X,Y)$ contains $2^{\continuum}$
  distinct closed ideals.
\end{thm}

\begin{proof}
  Choose  $k_1<l_1<k_2<l_2<\dots$ in $\bn$ with
  $\lim_{n\to\infty}k_n^{\max(1,p/2)} l_n^{-1}=0$ and unit vectors 
  $(x_{n,j})_{j=1}^{l_n}\subset\ell_2^{k_n}$ for all $n\in\bn$ to
  satisfy~\eqref{E:RIP1} and~\eqref{E:RIP2}. Let $E_n=\ell_2^{k_n}$
  and $F_n=\ell_\infty^{l_n}$ for all $n\in\bn$. As $E_n$ is a Hilbert
  space, we may take
  $(x^*_{n,j})_{j=1}^{l_n}=(x_{n,j})_{j=1}^{l_n}\subset
  S_{E^*_n}$. Suppose that $C_1,C_2>0$ are such that if
  $(x_n)_{n=1}^\infty\in X$ then
  $\big(\sum \norm{x_n}^{p}\big)^{1/p}\leq C_1\norm{(x_n)}_X$ and if
  $(y_n)_{n=1}^\infty\in Y$ then
  $\norm{(y_n)}_Y\leq C_2\big(\sum\norm{y_n}^{q}\big)^{1/q}$.
  
  For each $n\in\bn$ we define the operator
  $T_n\colon\ell_2^{k_n}\to\ell_\infty^{l_n}$ by
  $x\mapsto(\ip{x}{x_{n,j}})_{j=1}^{l_n}$. We now show that the
  conditions of Proposition~\ref{P:2.2.alt} are satisfied.

  We have that~\eqref{E:2.2a} is satisfied as if $(x_n)\in X$ then
  \begin{align*}
    \bignorm{T \big((x_n)\big)}_Y&\leq C_2 \left(\sum \norm{T_n
      x_n}^q_{\infty}\right)^{1/q}\\
    &= C_2 \left(\sum \sup_{1\leq j\leq l_n}\abs{\ip{x_n}{
        x_{n,j}}}^q\right)^{1/q}\\
    &\leq C_2 \left(\sum \norm{x_n}^q\right)^{1/q}\\
    &\leq C_2 \left(\sum \norm{x_n}^p\right)^{1/p} \leq C_2
    C_1\norm{(x_n)}_X\ .
  \end{align*}
  Thus, the map $(x_n)\mapsto T\big((x_n)\big)$ is well-defined and
  bounded. Condition~\eqref{E:2.2b} is trivially satisfied as
  $(x^*_{m,i})_{i=1}^{l_m}=(x_{m,i})_{i=1}^{l_m}$ for all $m\in\bn$.
  
  To prove~\eqref{E:2.2d}, fix $n\in\bn$, and let $l\in\bn$ be such
  that $l\geq l_n>k_n$. Given $m\in\bn$ with $l_m\geq l$ and
  $A\subset\{1,2,\dots,l_m\}$ with $\abs{A}=l$, set
  $t_n=\ceil{l/k_n}$.  Partition $A$ into $(A_j)_{j=1}^{t_n}$ such
  that $\abs{A_j}\leq k_n$ for all
  $1\leq j\leq t_n$. By~\eqref{E:RIP1} we have for all
  $1\leq j\leq t_n$ that
  \[
  \Bignorm{\sum_{i\in A_j}\sigma_i x_{m,i}}^2\leq 2
  \abs{A_j}\qquad\text{for all }(\sigma_i)_{i\in A_j}\subset
  \{\pm1\}\ .
  \]
  Thus, for all $(\sigma_i)_{i=1}^l\subset\{\pm1\}$ we have that
  \begin{align*}
    \Bignorm{\sum_{i\in A} \sigma_i x_{m,i}}&\leq \sum_{j=1}^{t_n}
    \Bignorm{\sum_{i\in A_j} \sigma_i  x_{m,i}}\\
    &\leq \sum_{j=1}^{t_n} 2^{1/2} \abs{A_j}^{1/2}\\
    &\leq t_n 2^{1/2} k_n^{1/2}\\
    &< (2l/k_n)2^{1/2} k_n^{1/2}<4lk_n^{-1/2}\ .
  \end{align*}
  Thus, for
  \[
  \varf_m(l)=\sup\Big\{\Bignorm{\sum_{i\in A} \sigma_i x_{m,i}}:\,
  A\subset\{1,2,\dots,l_m\},\ \abs{A}\leq l,\ (\sigma_i)_{i\in A}
  \subset\{\pm1\}\Big\}
  \]
  we have that $\frac{\varf_m(l)}{l}<4k_n^{-1/2}$. Hence,
  $\ds\lim_{l\to\infty} \sup_{m\in\bn,\ l_m\geq l}\frac{\varf_m(l)}l=0$,
  and we have~\eqref{E:2.2d}.
    
  We next verify condition~(a) of Proposition~\ref{P:2.2.alt}. Fix
  $\vare>0$.  There exists $\delta>0$ such that if $(a_j)\in\ell_p$
  with $\norm{(a_j)}_{\ell_p}\leq C_1$ and $\abs{a_j}\leq\delta$ for
  all $j\in\bn$ then $\norm{(a_j)}_{\ell_q}<C_2^{-1}\vare$. Let
  $x=(x_n)\in S_{X}$ such that
  $\sup_{1\leq j\leq l_n}\abs{\ip{x_n}{x_{n,j}}}\leq\delta$ for all
  $n\in\bn$. Thus, we have that
  \begin{equation}
    \label{E:upp}
    \Big(\sum_{n=1}^\infty \sup_{1\leq j\leq l_n}
    \abs{\ip{x_n}{x_{n,j}}}^p\Big)^{1/p}\leq (\sum_{n=1}^\infty
    \norm{x_n}^p)^{1/p}\leq C_1
  \end{equation}
  and
  \begin{align*}
    \bignorm{T\big((x_n)\big)}_Y&\leq C_2 \Big(\sum
    \norm{T_nx_n}^q_{\infty}\Big)^{1/q}\\
    &= C_2 \Big(\sum \sup_{1\leq j\leq
      l_n}\abs{\ip{x_n}{x_{n,j}}}^q\Big)^{1/q}\\
    <\vare \qquad\text{by~\eqref{E:upp} and our assumption on
      $\delta$}.
  \end{align*}
  Finally, it follows from~\eqref{E:RIP1} and~\eqref{E:RIP2} that the
  conditions of Lemma~\ref{L:RIP} are satisfied for $s=2$ and
  $t=p$. This in turn implies assumption~(b) of
  Proposition~\ref{P:2.2.alt}, and thus the proof is complete.
\end{proof}

\begin{rem}
  The earlier remark following the proof of
  Theorem~\ref{thm:ufdds-with-lower-upper-estimates} applies here,
  too. The closed ideals constructed above are all contained in the
  ideal of finitely strictly singular operators.
\end{rem}

Theorem~\ref{T:UP} gives the following immediate corollary.

\begin{cor}
  Let $1<p<\infty$. Then $\cL(\ell_p,\co)$, $\cL(\ell_1,\ell_p)$, and
  $\cL(\ell_p,\ell_\infty)$ each contain $2^{\continuum}$ distinct
  closed ideals.
\end{cor}

\begin{proof}
  We have by Theorem~\ref{T:UP} that
  $\cL\big(\big(\bigoplus_{n=1}^\infty \ell_2^{n}\big)_{\ell_p},\co\big)$
  contains $2^{\continuum}$ distinct closed ideals, and
  $\big(\bigoplus_{n=1}^\infty\ell_2^{n}\big)_{\ell_p}$ is isomorphic
  to $\ell_p$ for $1<p<\infty$. By duality we have that
  $\cL(\ell_1,\ell_p)$ and $\cL(\ell_p,\ell_\infty)$ each contain
  $2^{\continuum}$ distinct closed ideals.
\end{proof}

In the previous section we deduced from our results that the
cardinality of the lattice of closed ideals of $\cL(L_p)$, $1<p\neq
2<\infty$, is $2^{\continuum}$. Note that the 
Hardy space $H_1$ and its predual VMO can be seen as the
``well-behaved'' limit cases of the $L_p$-spaces. For example $\ell_2$
is complemented in both spaces, and $H_1$ contains a complemented copy
of $\ell_1$ and VMO a complemented copy of $\co$
(\textit{cf.}~\cite{M1} and~\cite{M2}*{page~125}). Thus, we deduce
the following corollary.

\begin{cor}
  The  cardinality of the lattice of closed ideals of
  $\cL(\text{VMO})$ and $\cL(H_1)$ is $2^{\continuum}$.
\end{cor}

\section{Final remarks and open problems}

If one only considers Banach spaces $X$ with an unconditional basis
or, more generally, with an UFDD, then the cardinalities $\kappa$ for
which we know examples of Banach spaces $X$ with an UFDD for which
the number of non-trivial proper closed ideals of $\cL(X)$ is $\kappa$,
are only the following three: 
\begin{mylist}{$\kappa=2^\continuum$}
\item[$\kappa=1$]
  For $X=\ell_p$, $1\leq p<\infty$, or $X=c_0$, the closed ideal of
  compact operators is the only non-trivial proper closed
  ideal~\cite{goh-mar-fel:60}.
\item[$\kappa=2$]
  For the spaces $X=\big(\bigoplus\ell_2^n\big)_{c_0}$ and its dual
  $X^*=\big(\bigoplus\ell_2^n\big)_{\ell_1}$, there are exactly two
  non-trivial proper closed ideals, the compacts and the closure of
  operators which factor through $c_0$ or $\ell_1$,
  respectively~\cites{laus-loy-read:04,laus-schlump-zsak:06}.
\item[$\kappa=2^\continuum$]
  $\cL(X)$ has $2^\continuum$ closed ideals for the spaces listed in
  the previous  two sections. In addition to these spaces, it was
  recently observed by Johnson 
  that also $\cL(T)$, where $T$ is Tsirelson space, has
  $2^{\continuum}$ closed ideals.
\end{mylist}
This begs the   following  questions:

\begin{problem}
  \label{prob:spaces-with-ufdd}
  Are there Banach spaces $X$ with an unconditional basis or
  unconditional UFDD for which the cardinality of the non-trivial
  proper closed ideals of $\cL(X)$ is strictly between $2$ and
  $2^\continuum$? Can this cardinality be any natural number,
  countable infinite or $\continuum$?
\end{problem}
An interesting space in the context of this question is
$c_0\oplus\ell_1$. According to~\cite{sirotkin-wallis:16},
$\cL(c_0\oplus\ell_1)$ contains an $\omega_1$-chain of closed ideals.
\begin{problem}
  \label{prob:c_0-plus_l_1}
  What is the cardinality of the lattice closed ideals of
  $\cL(c_0\oplus\ell_1)$?
\end{problem}
An other space of interest for Problem~\ref{prob:spaces-with-ufdd} is
the Schreier space. In~\cite{bkl:20} it was shown that the space of
operators on this space has continuum many maximal ideals.
\begin{problem}
\label{prob:schreier}
  What is the cardinality of the lattice of closed ideals of the space
  of operators on Schreier space?
\end{problem}

Among the class of general separable Banach spaces, there are more
examples for which the lattice of closed ideals of their algebra of
operators, or at least its cardinality, is determined. Such a list can
be found in~\cite{kania-lausten:17}.
 
Based on the construction by Argyros and
Haydon~\cite{argyros-haydon:11} of a space on which all operators are
compact perturbations of multiples of the identity,
Tarbard~\cite{tarbard:12} constructed for each $n\in\bn$ a space
$X_n$ for which $\cL(X_n)$ has exactly $n$ non-trivial proper closed
ideals. There are also Banach spaces $X$ for which the cardinality of
the lattice of closed ideals of $\cL(X) $ is exactly
$\continuum$. Indeed, suppose that $A$ is a separable Banach algebra
isomorphic to the Calkin algebra $\cL(X)/\cK(X)$ for a Banach space
$X$ which has the approximation property (to ensure $\cK(X)$ is the
smallest non-trivial closed ideal). Then, as observed
in~\cite{kania-lausten:17}, the closed ideals of $\cL(X)$ arise from
preimages of closed ideals in $A$. Examples of separable Banach spaces
$X$ for which $\cL(X)/\cK(X)$ has exactly continuum many closed ideals
were constructed, for instance,
in~\cite{motakis-puglisi-zisi:16} and~\cite{tarbard:13}. An
example of a space $X$ for which the number of closed ideals is
infinite but countable seems to be missing.
\begin{problem}
  \label{prob:aleph_0}
  Are there Banach spaces $X$ for which $\cL(X)$ has countably
  infinitely many closed ideals?
\end{problem}
A candidate of such a space is $C[0,\alpha]$, where $\alpha$ is a
large enough countable ordinal. But already for $\alpha=\omega^\omega$
(the first ordinal $\alpha$ for which $C[0,\alpha]$ is not isomorphic
to $c_0$) the answer of the following question is not known.
\begin{problem}
  \label{prob:cont-fns-on-countable-ordinal}
  For a countable ordinal $\alpha$, what are closed ideals of
  $\cL\big(C[0,\alpha]\big)$? What is the cardinality of the lattice
  of these ideals?
\end{problem}
Another space of interest is $L_1[0,1]$. It was shown by Johnson,
Pisier and Schechtman~\cite{john-pis-schec:20} that
$\cL(L_1[0,1])$ has at least $\continuum$ closed ideals.
\begin{problem}
  \label{prob:L_1}
  How many closed ideals does $\cL(L_1[0,1])$ have?
\end{problem}
As alluded by our terminology, the closed ideals of $\cL(X,Y)$ for
Banach spaces $X$ and $Y$ form a lattice with respect to inclusion and
with lattice operations given by $\cI\meet\cJ=\cI\cap\cJ$ and
$\cI\join\cJ=\cl{\cI+\cJ}$ for closed ideals $\cI$ and $\cJ$. In the
problems above, we have only asked about the cardinality of this
lattice. As a future ambitious target, one could study the lattice
structure.

\end{document}